\documentclass[12pt]{amsart}
\usepackage{amssymb,latexsym}
\usepackage{amsfonts}
\usepackage{amsmath}
\usepackage[colorlinks,linkcolor=blue,anchorcolor=blue,citecolor=blue]{hyperref}

\newcommand\F{{\mathbb F}}

\newcommand\R{{\mathbb R}}
\newcommand\Z{{\mathbb Z}}

\newcommand\x{{\bf x}}
\newcommand\vb{{\bf v}}
\newcommand\wb{{\bf w}}
\newcommand\spa{{\rm span}}
\newcommand\pa{{\rm Part\,}}
\newcommand\fand{ \quad \textrm{and} \quad} 
\newcommand\ffor{ \quad \textrm{or} \quad}

\newcommand\OO{{\mathcal O}}

\newcommand\PP{{\mathcal P}}
\newcommand\Lp{L_{\mathcal P}}
\newcommand\Pb{{\bf P}}
\newcommand\Qb{{\bf Q}}
\newcommand\Ob{{\bf O}}

\newtheorem{theorem}{Theorem}[section]
\newtheorem{lemma}[theorem]{Lemma}
\newtheorem{proposition}[theorem]{Proposition}
\newtheorem{corollary}[theorem]{Corollary}

\theoremstyle{definition}

\theoremstyle{remark}

\numberwithin{equation}{section}

\begin{document}

\title[lattices from elliptic curves]{On the lattices from elliptic curves over finite fields}

\author{Min Sha}

\address{School of Mathematics and Statistics, University of New South Wales,
 Sydney NSW 2052, Australia}
\email{shamin2010@gmail.com}






\keywords{Elliptic curve, lattice, minimal vector, basis, covering radius\\ School of Mathematics and Statistics, University of New South Wales,
 Sydney NSW 2052, Australia. Phone: (+61)-0411840234}

\begin{abstract}
In this paper, we continue the recent work of Fukshansky and Maharaj on  lattices from elliptic curves over finite fields.  We show that there exist bases formed by minimal vectors for these lattices except only one case. We also compute their determinants, and obtain sharp bounds for the covering radius.
\end{abstract}

\maketitle

\section{Introduction}

Let $L\subseteq \R^n$ be a lattice of rank $k\le n$, where $\R^n$ is the usual $n$-dimensional row vector space over $\R$. The \textit{minimal distance} of $L$ is defined as
$$
d(L)=\min\{\|\x\|:\x\in L, \x\ne 0\},
$$
where $\|~\|$ is the usual Euclidean norm in $\R^n$. A vector $\x \in L$ is called a \textit{minimal vector} of $L$ if $\|\x\|=d(L)$. Suppose that $\{\vb_1,\dots,\vb_k\}$ is a basis of $L$, then its \textit{generator matrix} is 
\begin{equation*}
B = \left(
\begin{array}{ccc}
\vb_1  \\
\vb_2 \\
\vdots \\
\vb_k
\end{array} \right).
\end{equation*}
The \textit{determinant} of $L$ is
$$
\det L=\sqrt{\det(BB^{T})},
$$
where $T$ stands for the transpose.

Let $V:=\spa_{\R} L$ be the $k$-dimensional subspace of $\R^n$ spanned by $L$. The proportion of $V$ filled by the balls centered at points of $L$ and of radius $d(L)/2$ is called the \textit{packing density} of $L$, denoted by $\Delta(L)$. That is,
$$
\Delta(L)=\frac{\pi^{k/2}d(L)^k}{2^k\Gamma(\frac{k}{2}+1)\det L },
$$
where $\Gamma$ is the gamma function.

Rosenbloom and Tsfasman \cite{Tsfasman1990} discovered a construction (\textit{function field lattices}) of lattices from algebraic curves over finite fields to construct the so-called asymptotically good families of lattices; see also \cite[pp. 578--583]{Tsfasman1991}. This kind of lattices was independently discovered by Noam Elkies (written in an unpublished note) and was also studied by Quebbemann \cite{Quebbemann}. We recall this construction briefly as follows.

Let $X$ be a smooth proper curve of genus $g$ over a finite field $\F_q$ of $q$ elements, and $K=\F_q(X)$. Fix a non-empty subset $\PP \subseteq X(\F_q)$, where $X(\F_q)$ is the set of $\F_q$-rational points of $X$, and let $\OO_\PP^*$ denote the set of non-zero rational functions $f\in K$ whose divisors have support contained in $\PP$.

Here, for any rational point $\Pb$ on $X$, we denote its corresponding divisor in non-bold font $P$. We fix this notation correspondence throughout the paper.

Suppose that $|\PP|=n$ and $\PP=\{\Pb_0,\Pb_1,\ldots,\Pb_{n-1}\}$. For each $\Pb_i\in \PP$, let $v_i$ denote the corresponding normalized discrete valuation.
 For every $f\in \OO_\PP^*$, the corresponding principal divisor is
 $$
 (f)=\sum_{i=0}^{n-1}v_i(f)P_i,
 $$
 and we define
$$
\deg f=\sum_{v_i(f)>0}v_i(f)=\frac{1}{2}\sum_{i=0}^{n-1}|v_i(f)|.
$$
 Define the homomorphism
\begin{equation}\label{lattice}
\phi_{\PP}: \OO_\PP^* \to \Z^n, \quad f \mapsto (v_0(f),v_1(f),\ldots,v_{n-1}(f)).
\end{equation}
We denote by $\Lp$ the image of $\phi_{\PP}$. Then, $\Lp$ is a finite-index sublattice of the root lattice
$$
A_{n-1}=\left\{\x=(x_0,x_1,\ldots,x_{n-1})\in \Z^n: \sum_{i=0}^{n-1}x_i=0 \right \}
$$
with rank $n-1$ and
\begin{align*}
& d(L_\PP)\ge \min\{\sqrt{2\deg f}: f\in \OO_\PP^*\setminus \F_q\}, \\
& \det L_\PP \le \sqrt{n}|J_X(\F_q)|\le \sqrt{n}\left(1+q+\frac{|X(\F_q)|-q-1}{g}\right)^g,
\end{align*}
where $J_X$ is the Jacobian of $X$.

In \cite{Lenny2014}, Fukshansky and Maharaj considered the case that $E:=X$ is an elliptic curve over $\F_q$ and $\PP=E(\F_q)$, they determined the minimal distance, the minimal vectors, and the number of minimal vectors (actually in \cite[Theorem 3.2]{Lenny2014}, the parameter $\epsilon$ should be the number of 2-torsion points of $E$ contained in  $E(\F_q)$); furthermore, they proved that $\Lp$ is generated by its minimal vectors when $|\PP|\ge 5$. We will recall their results briefly in Section \ref{preliminary}.

Now, one interesting question is that whether $\Lp$ has a basis of minimal vectors when $|\PP|\ge 5$. The main contribution of this paper is to confirm that $\Lp$ has such a basis. Here, we want to indicate that there indeed exist lattices generated by minimal vectors which have no basis of minimal vectors; see \cite{Conway1995}.

In this paper, we continue the research of \cite{Lenny2014}. First, we point out that if $\PP$ is only a subgroup of $E(\F_q)$, then all the results related to $\Lp$ in \cite{Lenny2014} still hold. Especially, when $\PP$ is a cyclic group and put $n=|\PP|$, $\Lp$ is exactly the so-called \textit{Barnes lattice} \cite{Barnes}, which is denoted by $B_{n-1}$ and characterized by the following equations:
\begin{equation} \label{eq:Barnes}
\sum_{i=0}^{n-1}x_i=0 \quad \textrm{and} \quad \sum_{i=0}^{n-1}ix_i \equiv 0 \mod n.
\end{equation}

Let $\PP$ be a subgroup of $E(\F_q)$. We mainly show that $\Lp$ has a basis formed by minimal vectors except that $\PP$ is a cyclic group and $|\PP|=4$. In addition, we  compute its determinant, and obtain sharp bounds for its covering radius.

Most recently, the authors in \cite{Bottcher} generalize the lattice $\Lp$ coming from elliptic curves over finite fields to a construction of lattices from finite Abelian groups, which includes the current case we study. By using different arguments from here, they show that these lattices have a basis of minimal vectors except that the group is cyclic and of order four. They also improve the upper bound of covering radius presented here in the case of the Barnes lattices. Especially, they investigate the automorphism groups of such lattices.

\section{Preliminaries}\label{preliminary}

We first fix some notation for the rest of the paper. Let $E$ be an elliptic curve defined over a finite field $\F_q$, the point at infinity is denoted by $\Ob$. Equipped with the group law, $E(\F_q)$ is an abelian group with the identity element $\Ob$. Let $\PP$ be a subgroup of $E(\F_q)$ with
$|\PP|=n$. As \eqref{lattice}, we define a lattice $\Lp$ with rank $n-1$.
Here, we always assume that $n\ge 2$.

Although Fukshansky and Maharaj \cite{Lenny2014} only considered the case $\PP=E(\F_q)$, their results are valid for any subgroup $\PP$ of $E(\F_q)$. Because  their arguments are based on two facts: $E$ is an elliptic curve, and $\PP$ is a subgroup of $E(\F_q)$. One key property of elliptic curves used repeatedly is as follows; see \cite[Chapter III, Corollary 3.5]{Silverman}.

\begin{theorem}\label{principal}
Let $D=\sum_{P\in E} n_P P$ be a divisor of $E$. Then, $D$ is a principal divisor if and only if $\sum_{P\in E} n_P=0$ and $\sum_{P\in E} n_P \Pb=\Ob$.
\end{theorem}

In the sequel, as in \cite{Lenny2014, Tsfasman1991}, we identify $\Z^n$ with the set of all divisors of $E$ with support in $\PP$. We will often make use of this identification when working with lattice vectors by working with the corresponding divisors instead.

For the convenience of the reader, we restate some results from \cite{Lenny2014} concerning the lattice $\Lp$.

\begin{lemma}[\cite{Lenny2014}]\label{Lenny1}
Suppose that $n\ge 4$, then the minimal distance of $\Lp$ is 2, and the minimal vectors of $\Lp$ are of the form $P+Q-R-S$, where $\Pb,\Qb,{\bf R},{\bf S}\in \PP$ are distinct and $\Pb+\Qb={\bf R}+{\bf S}$. If $n=2$, then the minimal distance of $\Lp$ is $2\sqrt{2}$, and the minimal vectors are of the form $\pm (2P-2O)$, where $\PP=\{ \Pb,\Ob \}$. If $n=3$, then the minimal distance of $\Lp$ is $\sqrt{6}$, and the   minimal vectors are of the form $\pm (P+Q-2O), \pm (P-2Q+O)$, and $\pm (-2P+Q+O)$, where $\PP=\{ \Pb,\Qb,\Ob \} $.
\end{lemma}

\begin{theorem}[\cite{Lenny2014}]\label{Lenny2}
If $n\ge 5$, then the lattice $\Lp$ is generated by its minimal vectors.
\end{theorem}

Now, we present some simple arguments on the cases $n=2,3,4$, which can make Theorem  \ref{Lenny2} more complete and will be used afterwards.

\begin{lemma}\label{234}
If $n=2,3$, $\Lp$ is generated by its minimal vectors. If $n=4$ and $\PP$ is a non-cyclic group, then $\Lp$ is also generated by its minimal vectors. If $n=4$ and $\PP$ is a cyclic group, then $\Lp$ is not generated by its minimal vectors.
\end{lemma}
\begin{proof}
For $n=2$, $\PP=\{\Pb,\Ob\}$ where $2\Pb=\Ob$. By Theorem \ref{principal}, it is easy to see that $\Lp=\{2kP-2kO:k\in\Z\}$. So, it is easy to see that $\Lp$ can be generated by minimal vectors.

For $n=3$, $\mathcal{P}=\{\Pb,\Qb,\Ob\}$ where $\Qb=2\Pb$. By Theorem \ref{principal},  every $aP+bQ+cO\in L_\mathcal{P}$ has a form that $a=-2b-3s, c=b+3s$ for some integers $s$. Then
$$
aP+bQ+cO=(b+s)(-2P+Q+O)-s(P+Q-2O).
$$
So by Lemma \ref{Lenny1}, $L_\mathcal{P}$ can also be generated by minimal vectors.

If $n=4$ and $\mathcal{P}$ is non-cyclic, then we let $\mathcal{P}=\{\Pb_1,\Pb_2,\Pb_3,\Ob\}$, where $\Pb_1+\Pb_2=\Pb_3$ and $2\Pb_i=\Ob$ for $i=1,2,3$. Note that every $aP_1+bP_2+cP_3+dO\in L_\mathcal{P}$ has a form that $a=2r+1,b=2s+1,c=2t+1,d=-2r-2s-2t-3$, or $a=2r,b=2s,c=2t,d=-2r-2s-2t$,  for some integers $r,s,t$. Then, we have
\begin{align*}
& aP_1+bP_2+cP_3+dO\\
& =(r+s+1)(P_1+P_2-P_3-O)+(r+t+1)(P_1+P_3-P_2-O)\\
& \qquad \qquad +(s+t+1)(P_2+P_3-P_1-O),
\end{align*}
or
\begin{align*}
& aP_1+bP_2+cP_3+dO\\
& =(r+s)(P_1+P_2-P_3-O)+(r+t)(P_1+P_3-P_2-O)\\
& \qquad \qquad +(s+t)(P_2+P_3-P_1-O).
\end{align*}
Thus by Lemma \ref{Lenny1}, $L_\mathcal{P}$ can be generated by minimal vectors.

Finally, assume that $n=4$ and $\PP$ is a cyclic group, and then let $\mathcal{P}=\{\Pb_1,\Pb_2=2\Pb_1,\Pb_3=3\Pb_1,\Ob\}$. In view of the form of minimal vectors given in Lemma \ref{Lenny1}, it is easy to see that the minimal vectors of $\Lp$ are $\pm(\Pb_1+\Pb_2-\Pb_3-\Ob)$ and $\pm(\Pb_1-\Pb_2-\Pb_3+\Ob)$, whose rank is 2. However, the rank of $\Lp$ is 3. So, $\Lp$ cannot be generated by minimal vectors.
\end{proof}

\section{Another viewpoint of $\Lp$}\label{special}

In this section, we will see that actually $\Lp$ is a generalisation of the Barnes lattice.

Now, we assume that $\PP$ is a cyclic group, and let
$$
\PP=\{\Pb_0,\Pb_1,\ldots,\Pb_{n-1}\},
$$
where $\Pb_i=i\Pb$, $\Pb\in E(\F_q)$ is of order $n$. Note that $\Pb_0=\Ob$. Then, by \eqref{lattice} and Theorem \ref{principal}, 
we see that a vector $(x_0,x_1,\ldots,x_{n-1})\in \Lp$ if and only if
\begin{equation}\label{Barnes0}
\sum_{i=0}^{n-1}x_i=0 \quad \textrm{and} \quad \sum_{i=0}^{n-1}ix_i \equiv 0 \mod n,
\end{equation}
where the coordinate $x_i$ corresponds to the point $\Pb_i$ for $0\le i \le n-1$. So, $\Lp$ is exactly the Barnes lattice $B_{n-1}$, which has been defined in \eqref{eq:Barnes}; see \cite[Section 5.3]{Martinet} for a summary of its properties. Since $\PP \cong \Z/n\Z$, we can view $\Lp$ as a lattice associated to the group $\Z/n\Z$. In general, according to the group structure of $E(\F_q)$ (see Theorem \ref{Ruck}), the lattice $\Lp$ is actually associated to some group $\Z/m\Z \times \Z/n\Z$. Moreover, the authors in \cite{Bottcher} generalize this construction to finite Abelian groups.

When $\PP$ is a cyclic group, in \cite[Proposition 5.3.5]{Martinet} Martinet gave a basis of $\Lp$ formed by minimal vectors when $n=|\PP|\ge 7$; however this basis differs by the parity of $n$. Here, we will give a new such basis for $n\ge 5$ in a uniform way (see Theorem \ref{basis}), and in the next section we will generalize it to more general cases.

We first recall a basic lemma, which will be used repeatedly; see \cite[page 13, Corollary 2]{Cassels}, or \cite[page 20, Corollary]{Gruber}.
\begin{lemma}\label{Cassels}
Let $L$ be a lattice of rank $n$, and let $\vb_1,\ldots,\vb_{m}$ be linearly independent vectors of $L$. Then, there exists a basis $\{\wb_1,\ldots,\wb_n\}$ of $L$  such that
$$
\vb_i=a_{i1}\wb_1+\cdots+a_{ii}\wb_i
$$
with integers $0\le a_{ij}<a_{ii}$ for $1\le j<i$ and $i=1,2,\ldots,m$.
\end{lemma}

\begin{theorem}\label{basis}
Assume that $\PP$ is a cyclic subgroup of $E(\F_q)$ and $n=|\PP|\ge 5$. Define the following vectors in $\Z^n$
\begin{equation}\label{basis1}
\left\{\begin{array}{ll}
  \vb_1=(1,1,-1,0,0,\ldots,0,-1),\\
  \vb_2=(1,0,1,-1,0,\ldots,0,-1),\\
  \vdots \\
  \vb_{n-3}=(1,0,0,\ldots,0, 1, -1, -1),\\
  \vb_{n-2}=(-1,1,0,\ldots,0,1,-1,0),\\
   \vb_{n-1}=(-1,1,0,\ldots,0,0,1,-1).
\end{array}\right.
\end{equation}
Then, $\{\vb_1,\vb_2,\ldots,\vb_{n-1}\}$ is a basis of $\Lp$.
\end{theorem}
\begin{proof}
First, we note that each $\vb_i$ ($1\le i \le n-1$) is a minimal vector of $\Lp$. It is easy to see that if there exist real numbers $a_1,\ldots,a_{n-1}$ such that
$$
a_1\vb_1+a_2\vb_2+\cdots+ a_{n-1}\vb_{n-1}=0,
$$
we have $a_1=\cdots=a_{n-1}=0$. Thus,
$\vb_1,\ldots,\vb_{n-1}$ are linearly independent vectors of $\Lp$. By Lemma \ref{Cassels}, there exists a basis $\{\wb_1,\wb_2,\ldots,\wb_{n-1}\}$ of $\Lp$ such that
$$
\vb_i=a_{i1}\wb_1+\cdots+a_{ii}\wb_i
$$
with integers $0\le a_{ij}<a_{ii}$ $(j<i; i=1,2,\ldots,n-1)$.

Consider $\vb_1=a_{11}\wb_1$, since $\vb_1$ is a minimal vector, we have $a_{11}=1$ and $\vb_1=\wb_1$. Now, we consider $\vb_2=a_{21}\vb_1+a_{22}\wb_2$; comparing the fourth  entries of the vectors, we deduce that $a_{22}=1$, and thus $a_{21}=0$ and $\vb_2=\wb_2$. In fact, applying the same arguments, we obtain
$$
a_{ii}=1 \quad \textrm{and} \quad \vb_i=\wb_i \quad \textrm{for $1\le i \le n-3$}.
$$

Now, suppose $\wb_{n-2}=(x_0,x_1,\ldots,x_{n-1})$ and consider
$$
\vb_{n-2}=a_{n-2,1}\vb_1+\cdots+a_{n-2,n-3}\vb_{n-3}+a_{n-2,n-2}\wb_{n-2}.
$$
Comparing the first and the last entries of the vectors respectively, we obtain
\begin{equation*}
\left\{\begin{array}{ll}
 -1=a_{n-2,1}+\cdots +a_{n-2,n-3}+a_{n-2,n-2} x_0,\\
 0=-a_{n-2,1}-\cdots -a_{n-2,n-3}+a_{n-2,n-2} x_{n-1}.
\end{array}\right.
\end{equation*}
Then summing up the above two equations, we have $-1=a_{n-2,n-2} (x_0+x_{n-1})$, which implies that $a_{n-2,n-2}=1$. So, for $1\le j <n-2$ we have $a_{n-2,j}=0$. Thus 
$$
\vb_{n-2}=\wb_{n-2}.
$$

Finally, we suppose $\wb_{n-1}=(x_0,x_1,\ldots,x_{n-1})$. 
Consider
\begin{equation}\label{Vn1}
\vb_{n-1}=a_{n-1,1}\vb_1+\cdots+a_{n-1,n-2}\vb_{n-2}+a_{n-1,n-1}\wb_{n-1}.
\end{equation}
We put $b_j=a_{n-1,j}$ for $1\le j \le n-1$. 
By comparing the entries of the vectors in \eqref{Vn1}, we obtain the following system of linear equations:
\begin{equation}\label{system}
\left\{\begin{array}{ll}
 -1=b_{1}+b_{2}+\cdots +b_{n-3}-b_{n-2}+b_{n-1}x_0,\\
 1=b_{1}+b_{n-2}+b_{n-1}x_1,\\
 0=-b_{1}+b_{2}+b_{n-1}x_2,\\
 0=-b_{2}+b_{3}+b_{n-1}x_3,\\
 \vdots \\
 0=-b_{n-5}+b_{n-4}+b_{n-1}x_{n-4},\\
 0=-b_{n-4}+b_{n-3}+b_{n-2}+b_{n-1}x_{n-3},\\
 1=-b_{n-3}-b_{n-2}+b_{n-1}x_{n-2},\\
 -1=-b_{1}-b_{2}-\cdots -b_{n-3}+b_{n-1}x_{n-1}.
\end{array}\right.
\end{equation}
Here, we want to indicate that only the first two and the last three  equations from \eqref{system} show up when $n=5$. 
Notice that $0\le b_{j}<b_{n-1}$ for $1\le j<n-1$, which will be used repeatedly without indications.  
Considering the second equation: $1=b_{1}+b_{n-2}+b_{n-1}x_1$, we have $-1\le x_1 \le 1$. If $x_1=1$, then we must have $b_{n-1}=1$, which implies that 
$$
\vb_{n-1}=\wb_{n-1}.
$$
 So, to complete the proof it suffices to prove that $x_1=1$.

Applying the same arguments, we find that
$$
x_2=x_3=\cdots =x_{n-4}=0 \quad \textrm{and}  \quad x_{n-2}=1.
$$
Summing the two equations corresponding to $x_{n-3}$ and $x_{n-2}$, we have $1=-b_{n-4}+b_{n-1}(x_{n-3}+x_{n-2})$, which implies that $x_{n-3}+x_{n-2}=1$. So, we have $x_{n-3}=0$. Moreover, the equation corresponding to $x_{n-1}$ implies that $-1\le x_{n-1} \le n-4$.

Assume that $x_1=-1$. By Theorem \ref{principal}, we have
\begin{equation}\label{point}
x_0\Pb_0 + x_1\Pb_1 + \cdots + x_{n-1}\Pb_{n-1} = \Ob,
\end{equation}
that is
$$
x_1+2x_2+\cdots +(n-1)x_{n-1} \equiv 0 \mod n.
$$
Substituting relevant formulas, we get $x_{n-1}+3 \equiv 0 \mod n$, which contradicts with $2 \le x_{n-1}+3 \le n-1$. So, we have $x_1\ne -1$. Similarly, we can show that $x_1\ne 0$. Thus, we must have $x_1 = 1$. 
This completes the proof of the theorem.
\end{proof}

Here, we want to remark that in the proof of Theorem \ref{basis}, we don't use properties of $\wb_i$ like \eqref{point} when we prove $\vb_i=\wb_i$ for $1\le i \le n-2$. This observation is crucial for the next section.

In addition, another approach to prove Theorem \ref{basis} (and the coming Theorems \ref{basis2}, \ref{basis3}, \ref{basis4} and \ref{basism}) is to compute the determinant of the given vectors $\{\vb_1,\ldots,\vb_{n-1}\}$ and check whether it is equal to $\det \Lp$, which has been successfully performed in \cite{Bottcher}.

\section{Basis}

In this section, we want to prove the following theorem case by case.

\begin{theorem}\label{main}
Assume that $|\PP|\ge 5$. Then, $\Lp$ has a basis of minimal vectors.
\end{theorem}

Actually, we will construct explicit bases in the proof. Combining Theorem \ref{main} with Lemma \ref{234} and its proof, we can see that for the lattice $\Lp$ attached to the elliptic curve $E$, it has a basis formed by minimal vectors except that $\PP$ is a cyclic group and $|\PP|=4$.

\begin{theorem}\label{main2}
Assume that $|\PP|\ge 2$. Then, $\Lp$ has a basis of minimal vectors except that $\PP$ is a cyclic group and $|\PP|=4$.
\end{theorem}

We first recall a well-known result about the group structure of $E(\F_q)$ , which was first found by Tsfasman \cite{Tsfasman1985} and then independently proved by  R\"uck \cite{Ruck1987} and Voloch \cite{Voloch} following the paper of Schoof \cite{Schoof} (see the historical notes in \cite[Chapter 3]{Tsfasman2007}).

\begin{theorem}\label{Ruck}
Let $E$ be an elliptic curve defined over $\F_q$. Then, there exist positive integers $n_1,n_2$ such that 
$$
E(\F_q)\cong \Z/n_1\Z \times \Z/n_2\Z
$$
with $n_1|n_2$ and $n_1|q-1$.
\end{theorem}

Theorem \ref{Ruck} implies that every subgroup $\PP$ of $E(\F_q)$ can be described as
\begin{equation}\label{group}
\PP \cong \Z/m\Z \times \Z/n\Z
\end{equation}
with $1\le m \le n$. Thus, the lattice $\Lp$ is also associated to the group $\Z/m\Z \times \Z/n\Z$.

In the sequel, we will identify $\PP$ with the group $\Z/m\Z \times \Z/n\Z$ with $1\le m \le n$; for any integer $k\ge 1$, we also identify the group $\Z/k\Z$ with $\{0,1,2,\ldots,k-1\}$ in the sense of modulo $k$.

If $\gcd(m,n)=1$, then actually $\PP$ is a cyclic group, which has been discussed in Section \ref{special}. Here, we will construct a basis for $\Lp$ by using minimal vectors in more general settings.

In this section, the entries of a vector $\x=(x_0,x_1,\ldots,x_{mn-1})\in \Lp$ correspond to the following order of points in $\PP$:
\begin{align*}
&(0,0),(0,1),\ldots,(0,n-1),(1,0),(1,1),\ldots,(1,n-1),\ldots,\\
& \qquad \qquad (m-1,0),(m-1,1),\ldots,(m-1,n-1).
\end{align*} 

Now, we start to prove Theorem \ref{main} case by case. 

\begin{theorem}\label{basis2}
 Assume that $\PP$ satisfies $\PP \cong \Z/2\Z \times \Z/n\Z$ with $n\ge 5$. Then, $\Lp$ has a basis of minimal vectors.
\end{theorem}
\begin{proof}
We divide the points of $\PP$ into two parts as follows:
\begin{equation*}
\left\{\begin{array}{ll}
  \pa 0: (0,0),(0,1),\ldots,(0,n-1);\\
  \pa 1: (1,0),(1,1),\ldots,(1,n-1).
\end{array}\right.
\end{equation*}

We first associate to Part 0 the $n-1$ minimal vectors $\vb_1,\vb_2,\ldots,\vb_{n-1}$ (labelled in the same order) as described in Theorem \ref{basis}, and then extend them to be minimal vectors of $\Lp$ by setting the other $n$ entries zero.

We also associate to Part 1 the $n-2$ minimal vectors $\vb_{n},\vb_{n+1},\ldots,\vb_{2n-3}$ (labelled in the same order) as the first $n-2$ minimal vectors described in Theorem \ref{basis}, and then extend them to be minimal vectors of $\Lp$ by setting the other $n$ entries zero.

Now, define the following minimal vectors of $\Lp$:
$$
\vb_{2n-2}=(\underbrace{0,1,-1,0,\ldots,0}_{\textrm{Part 0}}, \underbrace{0,-1,1,0,\ldots,0}_{\textrm{Part 1}})
$$
and
$$
\vb_{2n-1}=(\underbrace{0,1,1,0,\ldots,0}_{\textrm{Part 0}},   \underbrace{0,-1,-1,0,\ldots,0}_{\textrm{Part 1}}).
$$
We claim that the above minimal vectors $\vb_1,\vb_2,\ldots,\vb_{2n-1}$ form a basis of $\Lp$.

First, we prove that $\vb_1,\vb_2,\ldots,\vb_{2n-1}$ are linearly independent. Here (and also in the proofs of Theorems \ref{basis3}, \ref{basis4} and \ref{basism}), we use the fact that $\Z$-linear independence and $\R$-linear independence are equivalent for vectors in $\R^m$ with integer entries for any positive integer $m$. Suppose that there exist $a_1,a_2,\ldots,a_{2n-1}\in\Z$ such that
$$
a_1\vb_1+a_2\vb_2+\cdots+a_{2n-1}\vb_{2n-1}=0.
$$
For each vector $\vb_i$ ($1\le i \le 2n-1$), we consider the sum of its entries  corresponding to Part 0. This sum is zero if $1\le i \le 2n-2$. So, we have $a_{2n-1}=0$. Now, we consider 
$$
a_1\vb_1+a_2\vb_2+\cdots+a_{2n-2}\vb_{2n-2}=0.
$$
Considering the entries of $\vb_i$ ($1\le i \le 2n-2$) corresponding to Part 1, we find that this part of $\vb_{2n-2}$ corresponds to the point $(0,1)$, while this part of $\vb_i$ ($1\le i \le 2n-3$) corresponds to the point $(0,0)$. So, we must have $a_{2n-2}=0$. Clearly by Theorem \ref{basis}, $\vb_1,\vb_2,\ldots,\vb_{2n-3}$ are linearly independent, this completes the proof of linear independence.

By Lemma \ref{Cassels}, there exists a basis $\{\wb_1,\wb_2,\ldots,\wb_{2n-1}\}$ of $\Lp$ such that
$$
\vb_i=a_{i1}\wb_1+\cdots+a_{ii}\wb_i
$$
with integers $0\le a_{ij}<a_{ii}$ $(j<i; i=1,2,\ldots,2n-1)$. We first consider the entries corresponding to Part 0 for $\vb_1,\ldots,\vb_{n-1}$, and then consider the entries corresponding to Part 1 for $\vb_{n},\ldots,\vb_{2n-3}$.  Applying the same arguments as in the proof of Theorem \ref{basis}, we can obtain
$$
a_{ii}=1 \quad \textrm{and} \quad \vb_i=\wb_i \quad \textrm{for $1\le i \le 2n-3$}.
$$

Now, we consider 
\begin{equation}\label{V2n1}
\vb_{2n-2}=a_{2n-2,1}\vb_1+\cdots+a_{2n-2,2n-3}\vb_{2n-3}+a_{2n-2,2n-2}\wb_{2n-2}.
\end{equation}
We define $b_j=a_{2n-2,n-1+j}$ for $1\le j \le n-1$; in particular, $b_{n-1}=a_{2n-2,2n-2}$. Suppose that the entries of $\wb_{2n-2}$ corresponding to Part 1 are $(x_0,x_1,\ldots,x_{n-1})$. 
By comparing the entries of the vectors corresponding to Part 1 in \eqref{V2n1}, we obtain the following system of linear equations with integer coefficients and integer variables similar as \eqref{system}:
\begin{equation}\label{S2n1}
\left\{\begin{array}{ll}
 0=b_1+b_2+\cdots +b_{n-3}-b_{n-2}+b_{n-1}x_0,\\
 -1=b_{1}+b_{n-2}+b_{n-1}x_1,\\
 1=-b_{1}+b_{2}+b_{n-1}x_2,\\
 0=-b_{2}+b_{3}+b_{n-1}x_3,\\
 \vdots \\
 0=-b_{n-5}+b_{n-4}+b_{n-1}x_{n-4},\\
 0=-b_{n-4}+b_{n-3}+b_{n-2}+b_{n-1}x_{n-3},\\
 0=-b_{n-3}-b_{n-2}+b_{n-1}x_{n-2},\\
 0=-b_{1}-b_{2}-\cdots -b_{n-3}+b_{n-1}x_{n-1}.
\end{array}\right.
\end{equation}
Here, we want to indicate that only the first two and the last two equations from \eqref{S2n1} show up when $n=5$, and the equation corresponding to $x_2$ actually is $1=-b_1+b_2+b_3+b_4x_2$. When $n=6$, only the first three and the last three equations from \eqref{S2n1} appear.  These two individual cases can be handled similarly as the general case, and we can get  
$$
\vb_{2n-2}=\wb_{2n-2}
$$
when $n=5,6$.

Now, for $n\ge 7$ we need to solve the equation system \eqref{S2n1}. 
Notice that $0\le b_j <b_{n-1}$ for every $1\le j <n-1$, which will be used repeatedly without indications.  
Summing up the two equations corresponding to $x_0$ and $x_{n-1}$, we get $0=-b_{n-2}+b_{n-1}(x_0+x_{n-1})$, which implies that 
$$
x_0+x_{n-1}=0 \fand b_{n-2}=0.
$$ 
From the equation corresponding to $x_{n-2}$, we have $0=-b_{n-3}+b_{n-1}x_{n-2}$, which gives
$$
x_{n-2}=0 \fand b_{n-3}=0.
$$ 
Considering the equation corresponding to $x_{n-3}$, we have $0=-b_{n-4}+b_{n-1}x_{n-3}$, which yields
$$
x_{n-3}=0 \fand b_{n-4}=0.
$$ 
By the equations corresponding to $x_3,\ldots,x_{n-4}$ and noticing $b_{n-4}=0$, we find that
$$
x_3=\cdots =x_{n-4}=0 \fand b_2=\cdots=b_{n-4}=0.
$$
Then, the equation corresponding to $x_0$ becomes $0=b_1+b_{n-1}x_0$, which implies that 
$$
x_0=0 \fand b_1=0.
$$
 As a result, the equation corresponding to $x_1$ becomes $-1=b_{n-1}x_1$, thus we have $b_{n-1}=1$. So, we get  
$$
\vb_{2n-2}=\wb_{2n-2}.
$$

Finally, we consider 
\begin{equation}\label{V2n2}
\vb_{2n-1}=a_{2n-1,1}\vb_1+\cdots+a_{2n-1,2n-2}\vb_{2n-2}+a_{2n-1,2n-1}\wb_{2n-1},
\end{equation}
we define $b_j=a_{2n-1,n-1+j}$ for $1\le j \le n$; in particular, $b_{n}=a_{2n-1,2n-1}$. Suppose that the entries of $\wb_{2n-1}$ corresponding to Part 1 are $(x_0,x_1,\ldots,x_{n-1})$. 
By comparing the entries of the vectors corresponding to Part 1 in \eqref{V2n2}, we obtain the following system of linear equations with integer coefficients and integer variables:
\begin{equation}\label{S2n2}
\left\{\begin{array}{ll}
 0=b_1+b_2+\cdots +b_{n-3}-b_{n-2}+b_{n}x_0,\\
 -1=b_{1}+b_{n-2}-b_{n-1}+b_{n}x_1,\\
 -1=-b_{1}+b_{2}+b_{n-1}+b_{n}x_2,\\
 0=-b_{2}+b_{3}+b_{n}x_3,\\
 \vdots \\
 0=-b_{n-5}+b_{n-4}+b_{n}x_{n-4},\\
 0=-b_{n-4}+b_{n-3}+b_{n-2}+b_{n}x_{n-3},\\
 0=-b_{n-3}-b_{n-2}+b_{n}x_{n-2},\\
 0=-b_{1}-b_{2}-\cdots -b_{n-3}+b_{n}x_{n-1}.
\end{array}\right.
\end{equation}
Here, we want to indicate that only the first two and the last two equations from \eqref{S2n2} show up when $n=5$, and the equation corresponding to $x_2$ actually is $-1=-b_1+b_2+b_3+b_4+b_5x_2$. When $n=6$, only the first three and the last three equations from \eqref{S2n2} appear.  These two individual cases can be handled similarly as the  general case, and we can get 
$$
\vb_{2n-1}=\wb_{2n-1}
$$
when $n=5,6$.

Now, for $n\ge 7$ we need to solve the equation system \eqref{S2n2}.
Notice that $0\le b_j <b_n$ for every $1\le j <n$. 
Applying the same arguments exactly as solving \eqref{S2n1}, we get 
$$
x_0=x_3=x_4=\cdots=x_{n-1}=0
$$
and 
$$
b_1=b_2=\cdots=b_{n-2}=0.
$$
As a result, the equation corresponding to $x_2$ becomes $-1=b_{n-1}+b_nx_2$, thus we have 
$$
x_2=-1 \fand b_n=b_{n-1}+1.
$$
Now, the equation corresponding to $x_1$ becomes $-2=b_n(x_1-1)$. So, we have $b_n=1$ or 2. Suppose that $b_n=2$, then we have $x_1=0$, and thus the entries of $\wb_{2n-1}$ corresponding to Part 1 are $(0,0,-1,0,\ldots,0)$, which contradicts with $\wb_{2n-1}\in \Lp$. Indeed, notice that for any point in Part 0, its first component is 0; by regarding the first component we can see that $\wb_{2n-1}$ does not satisfy the condition of point addition in Theorem \ref{principal}, and thus we have $\wb_{2n-1}\not\in \Lp$. So, we must have $b_n=1$, that is 
$$
\vb_{2n-1}=\wb_{2n-1}, 
$$
which completes the proof of the theorem.
\end{proof}

\begin{theorem}\label{basis3}
 Assume that $\PP$ satisfies $\PP \cong \Z/3\Z \times \Z/n\Z$ with $n\ge 5$. Then, $\Lp$ has a basis of minimal vectors.
\end{theorem}
\begin{proof}
We divide the points of $\PP$ into three parts as follows:
\begin{equation*}
\left\{\begin{array}{ll}
  \pa 0: (0,0),(0,1),\ldots,(0,n-1);\\
  \pa 1: (1,0),(1,1),\ldots,(1,n-1);\\
  \pa 2: (2,0),(2,1),\ldots,(2,n-1).
\end{array}\right.
\end{equation*}

We first associate to Part 0 the $n-1$ minimal vectors $\vb_1,\vb_2,\ldots,\vb_{n-1}$ (labelled in the same order) as described in Theorem \ref{basis}, and then extend them to be minimal vectors of $\Lp$ by setting the other $2n$ entries zero.

For each Part $i$ ($i=1,2$) (first Part 1, and then Part 2), we also associate to Part $i$ the $n-2$ minimal vectors (labelled in the same order) as the first $n-2$ minimal vectors described in Theorem \ref{basis}, and then extend them to be minimal vectors of $\Lp$ by setting the other $2n$ entries zero. By this way, we obtain another $2(n-2)$ minimal vectors of $\Lp$: $\vb_{n},\vb_{n+1},\ldots,\vb_{3n-5}$ (listed via the indicated order).

Now, define the following minimal vectors of $\Lp$:
$$
\vb_{3n-4}=(\underbrace{0,1,-1,0,\ldots,0}_{\textrm{Part 0}}, \underbrace{0,-1,1,0,\ldots,0}_{\textrm{Part 1}},
\underbrace{0,0,0,0,\ldots,0}_{\textrm{Part 2}}),
$$
$$
\vb_{3n-3}=(\underbrace{0,1,-1,0,\ldots,0}_{\textrm{Part 0}}, \underbrace{0,0,0,0,\ldots,0}_{\textrm{Part 1}},
\underbrace{0,-1,1,0,\ldots,0}_{\textrm{Part 2}}),
$$
$$
\vb_{3n-2}=(\underbrace{1,0,0,0,\ldots,0}_{\textrm{Part 0}}, 
\underbrace{0,0,0,1,0,\ldots,0}_{\textrm{Part 1}},
 \underbrace{0,-1,-1,0,\ldots,0}_{\textrm{Part 2}}),
$$
and 
$$
\vb_{3n-1}=(\underbrace{1,0,0,0,\ldots,0}_{\textrm{Part 0}}, 
\underbrace{-1,-1,0,0,\ldots,0}_{\textrm{Part 1}},
 \underbrace{0,1,0,0,\ldots,0}_{\textrm{Part 2}}).
$$
We claim that the above minimal vectors $\vb_1,\vb_2,\ldots,\vb_{3n-1}$ form a basis of $\Lp$.

First, we prove that $\vb_1,\vb_2,\ldots,\vb_{3n-1}$ are linearly independent. Suppose that there exist $a_1,a_2,\ldots,a_{3n-1}\in\Z$ such that
$$
a_1\vb_1+a_2\vb_2+\cdots+a_{3n-1}\vb_{3n-1}=0.
$$
For each vector $\vb_i$ ($1\le i \le 3n-1$), we consider the sum of its entries corresponding to Part 0. This sum is zero if $1\le i \le 3n-3$. So, we have 
$a_{3n-2}+a_{3n-1}=0$. Similarly, by considering the sum of entries corresponding to Part 1, we get $a_{3n-2}-2a_{3n-1}=0$. So, we obtain $a_{3n-2}=a_{3n-1}=0$. Now, we consider 
$$
a_1\vb_1+a_2\vb_2+\cdots+a_{3n-3}\vb_{3n-3}=0.
$$
Considering the entries of $\vb_i$ ($1\le i \le 3n-3$) corresponding to Part 2, we find that this part of $\vb_{3n-3}$ corresponds to the point $(0,1)$, while this part of $\vb_i$ ($1\le i \le 3n-4$) corresponds to the point $(0,0)$. 
So, we must have $a_{3n-3}=0$.
Similarly, by comparing the entries corresponding to Part 1, we can obtain $a_{3n-4}=0$. Clearly, $\vb_1,\vb_2,\ldots,\vb_{3n-5}$ are linearly independent, this completes the proof of linear independence.

By Lemma \ref{Cassels}, there exists a basis $\{\wb_1,\wb_2,\ldots,\wb_{3n-1}\}$ of $\Lp$ such that
$$
\vb_i=a_{i1}\wb_1+\cdots+a_{ii}\wb_i
$$
with integers $0\le a_{ij}<a_{ii}$ $(j<i; i=1,2,\ldots,3n-1)$. From the proof of Theorem \ref{basis2}, it follows directly that
$$
\vb_i=\wb_i \quad \textrm{for $1\le i \le 3n-2$}.
$$ 
Now,  consider 
\begin{equation}\label{V3n1}
\vb_{3n-1}=a_{3n-1,1}\vb_1+\cdots+a_{3n-1,3n-2}\vb_{3n-2}+a_{3n-1,3n-1}\wb_{3n-1}.
\end{equation}
We let $a=a_{3n-1,3n-2}$ and $b=a_{3n-1,3n-1}$. We also suppose that the sums of entries of $\wb_{3n-1}$ corresponding to Part 1 and Part 2 are $x$ and $y$, respectively. Then by \eqref{V3n1}, we can get the following:
\begin{equation}\label{S4n1}
\left\{\begin{array}{ll}
 -2=a+bx,\\
 1=-2a+by.
\end{array}\right.
\end{equation}
 Notice that $0\le a<b$. From the second equation, we find that 
 $$
 y=1.
 $$
By the first equation, we get 
$$
x=-2 \ffor x=-1.
$$
If $x=-2$, then we must have $b=1$, which implies that 
$$
\vb_{3n-1}=\wb_{3n-1}.
$$
So, to complete the proof it remains to show that $x\ne -1$.

We assume that $x=-1$. Note that $y=1$. Considering the point addition related to $\wb_{3n-1}$ as in Theorem \ref{principal}, we can see that the sum is not $(0,0)$ by regarding the first component (actually the first component is 1), which contradicts with $\wb_{3n-1}\in \Lp$. So, we have $x\ne -1$. This completes the proof of the theorem. 
\end{proof}

\begin{theorem}\label{basis4}
 Assume that $\PP$ satisfies $\PP \cong \Z/4\Z \times \Z/n\Z$ with $n\ge 5$. Then, $\Lp$ has a basis of minimal vectors.
\end{theorem}
\begin{proof}
We divide the points in $\PP$ into four parts as follows:
\begin{equation*}
\left\{\begin{array}{ll}
  \pa 0: (0,0),(0,1),\ldots,(0,n-1);\\
  \pa 1: (1,0),(1,1),\ldots,(1,n-1);\\
  \pa 2: (2,0),(2,1),\ldots,(2,n-1);\\
  \pa 3: (3,0),(3,1),\ldots,(3,n-1).
\end{array}\right.
\end{equation*}

We first associate to Part 0 the $n-1$ minimal vectors $\vb_1,\vb_2,\ldots,\vb_{n-1}$ (labelled in the same order) as described in Theorem \ref{basis}, and then extend them to be minimal vectors of $\Lp$ by setting the other $3n$ entries zero.

For each Part $i$ ($i=1,2,3$) (first Part 1, then Part 2, and finally Part 3), we also associate to Part $i$ the $n-2$ minimal vectors (labelled in the same order) as the first $n-2$ minimal vectors described in Theorem \ref{basis}, and then extend them to be minimal vectors of $\Lp$ by setting the other $3n$ entries zero. By this way, we obtain another $3(n-2)$ minimal vectors of $\Lp$: $\vb_{n},\vb_{n+1},\ldots,\vb_{4n-7}$ (listed via the indicated order).

Now, define the following minimal vectors of $\Lp$:
\begin{align*}
\vb_{4n-6}=(\underbrace{0,1,-1,0,\ldots,0}_{\textrm{Part 0}}, \underbrace{0,-1,1,0,\ldots,0}_{\textrm{Part 1}},
\underbrace{0,0,\ldots,0}_{\textrm{Part 2}},
\underbrace{0,0,\ldots,0}_{\textrm{Part 3}}),
\end{align*}
$$
\vb_{4n-5}=(\underbrace{0,1,-1,0,\ldots,0}_{\textrm{Part 0}}, 
\underbrace{0,0,\ldots,0}_{\textrm{Part 1}},
\underbrace{0,-1,1,0,\ldots,0}_{\textrm{Part 2}},
\underbrace{0,0,\ldots,0}_{\textrm{Part 3}}),
$$
$$
\vb_{4n-4}=(\underbrace{0,1,-1,0,\ldots,0}_{\textrm{Part 0}}, 
\underbrace{0,0,\ldots,0}_{\textrm{Part 1}},
\underbrace{0,0,\ldots,0}_{\textrm{Part 2}},
\underbrace{0,-1,1,0,\ldots,0}_{\textrm{Part 3}}),
$$
$$
\vb_{4n-3}=(\underbrace{1,0,\ldots,0}_{\textrm{Part 0}}, 
\underbrace{0,0,\ldots,0}_{\textrm{Part 1}},
\underbrace{0,0,0,1,0,\ldots,0}_{\textrm{Part 2}},
 \underbrace{0,-1,-1,0,\ldots,0}_{\textrm{Part 3}}),
$$
$$
\vb_{4n-2}=(\underbrace{1,0,\ldots,0}_{\textrm{Part 0}}, 
\underbrace{0,-1,-1,0,\ldots,0}_{\textrm{Part 1}},
\underbrace{0,0,0,1,0,\ldots,0}_{\textrm{Part 2}},
\underbrace{0,0,\ldots,0}_{\textrm{Part 3}}),
$$
and 
$$
\vb_{4n-1}=(\underbrace{1,0,\ldots,0}_{\textrm{Part 0}}, 
\underbrace{-1,0,\ldots,0}_{\textrm{Part 1}},
\underbrace{-1,0,\ldots,0}_{\textrm{Part 2}},
 \underbrace{1,0,\ldots,0}_{\textrm{Part 3}}).
$$
We claim that the above minimal vectors $\vb_1,\vb_2,\ldots,\vb_{4n-1}$ form a basis of $\Lp$.

First, we prove that $\vb_1,\vb_2,\ldots,\vb_{4n-1}$ are linearly independent. Suppose that there exist $a_1,a_2,\ldots,a_{4n-1}\in\Z$ such that
$$
a_1\vb_1+a_2\vb_2+\cdots+a_{4n-1}\vb_{4n-1}=0.
$$
For each vector $\vb_i$ ($1\le i \le 4n-1$), we consider the sum of its entries corresponding to Part 0. This sum is zero if $1\le i \le 4n-4$. So, we have 
$a_{4n-3}+a_{4n-2}+a_{4n-1}=0$. Similarly, by considering the sum of entries corresponding to Part 1 and Part 2 respectively, we get $-2a_{4n-2}-a_{4n-1}=0$ and $a_{4n-3}+a_{4n-2}-a_{4n-1}=0$. So, we obtain $a_{4n-3}=a_{4n-2}=a_{4n-1}=0$. Besides, similar as before, we can get $a_{4n-4}=a_{4n-5}=a_{4n-6}=0$. Clearly, $\vb_1,\vb_2,\ldots,\vb_{4n-7}$ are linearly independent, this completes the proof of linear independence.

By Lemma \ref{Cassels}, there exists a basis $\{\wb_1,\wb_2,\ldots,\wb_{4n-1}\}$ of $\Lp$ such that
$$
\vb_i=a_{i1}\wb_1+\cdots+a_{ii}\wb_i
$$
with integers $0\le a_{ij}<a_{ii}$ $(j<i; i=1,2,\ldots,4n-1)$. From the proof of Theorem \ref{basis2}, it follows directly that
$$
\vb_i=\wb_i \quad \textrm{for $1\le i \le 4n-2$}.
$$ 
Now, consider 
\begin{equation}\label{V4n1}
\vb_{4n-1}=a_{4n-1,1}\vb_1+\cdots+a_{4n-1,4n-2}\vb_{4n-2}+a_{4n-1,4n-1}\wb_{4n-1}.
\end{equation}
We put $a=a_{4n-1,4n-3}, b=a_{4n-1,4n-2}$ and $c=a_{4n-1,4n-1}$. We also suppose that the sums of entries of $\wb_{4n-1}$ corresponding to Part 1, Part 2 and Part 3 are $x, y$ and $z$, respectively. Then by \eqref{V4n1}, we can get the following:
\begin{equation}\label{S4n1}
\left\{\begin{array}{ll}
 -1=-2b+cx,\\
 -1=a+b+cy,\\
 1=-2a+cz.
\end{array}\right.
\end{equation}
 Notice that $0\le a<c$ and $0\le b<c$. Considering the equation corresponding to $y$, we have $y=-1$ and $c=a+b+1$. Since $1=-2a+cz$, we find that $z=1$ and $c=2a+1$. Thus, we have 
 $$
 a=b \fand c=2b+1.
 $$
 Then, from the equation corresponding to $x$, we see that $2b+1$ is a divisor of $2b-1$. So, we must have $b=0$. Thus, we get $c=1$, which implies that 
 $$
 \vb_{4n-1}=\wb_{4n-1},
 $$ 
 which completes the proof of the theorem.
\end{proof}

\begin{theorem}\label{basism}
 Assume that $\PP$ satisfies $\PP \cong \Z/m\Z \times \Z/n\Z$ with $n\ge m \ge 5$. Then,  $\Lp$ has a basis of minimal vectors.
\end{theorem}
\begin{proof}
We divide the points of $\PP$ into $m$ parts as follows:
\begin{equation}\label{part1}
\left\{\begin{array}{ll}
  \pa 0: (0,0),(0,1),\ldots,(0,n-1);\\
  \pa 1: (1,0),(1,1),\ldots,(1,n-1);\\
  \vdots \\
  \pa m-1: (m-1,0),(m-1,1),\ldots,(m-1,n-1).
\end{array}\right.
\end{equation}

We first associate to Part 0 the $n-1$ minimal vectors $\vb_1,\vb_2,\ldots,\vb_{n-1}$ (labelled in the same order) as described in Theorem \ref{basis}, and then extend them to be minimal vectors of $\Lp$ by setting the other $(m-1)n$ entries zero.

For each Part $i$ ($1\le i \le m-1$) (the order is from 1 to $m-1$), we associate to Part $i$ the $n-2$ minimal vectors (labelled in the same order) as the first $n-2$ minimal vectors described in Theorem \ref{basis}, and then extend them to be minimal vectors of $\Lp$ by setting the other $(m-1)n$ entries zero. As a result, we get another $(m-1)(n-2)$ minimal vectors: $\vb_{n},\vb_{n+1},\ldots, \vb_{mn-2m+1}$ (listed via the indicated order).

Now, we define another $m-1$ minimal vectors of $\Lp$ as follows: first let the entries corresponding to Part 0 be $(0,1,-1,0,\ldots,0)$, then for each $i$ ($1\le i \le m-1$) (the order is from 1 to $m-1$), let the entries corresponding to Part $i$ be $(0,-1,1,0,\ldots,0)$, and put the other $(m-2)n$ entries zero. We denote these $m-1$ minimal vectors by $\vb_{mn-2m+2},\vb_{mn-2m+3},\ldots,\vb_{mn-m}$ (listed via the indicated order).

Then, we associate to the following $m$ points 
$$
(0,0),(1,0),(2,0),\ldots, (m-1,0)
$$
the $m-1$ minimal vectors (labelled in the same order) as described in Theorem \ref{basis}, and then extend them to be minimal vectors of $\Lp$ by setting the other $m(n-1)$ entries zero. We denote these $m-1$ minimal vectors by $\vb_{mn-m+1},\vb_{mn-m+2},\ldots,\vb_{mn-1}$ (listed via the indicated order).

Now, we claim that the above minimal vectors $\vb_1,\vb_2,\ldots,\vb_{mn-1}$ form a basis of $\Lp$.

First, we prove that $\vb_1,\vb_2,\ldots,\vb_{mn-1}$ are linearly independent. Suppose that there exist $a_1,a_2,\ldots,a_{mn-1}\in\Z$ such that
\begin{equation} \label{eq:linear}
a_1\vb_1+a_2\vb_2+\cdots+a_{mn-1}\vb_{mn-1}=0.
\end{equation}
Fixing $i$ ($0\le i \le m-1$), we consider the sum of entries of the vector $\vb_j$ $(1\le j \le mn-1)$ corresponding to Part $i$, this sum is zero if $1\le j \le mn-m$; furthermore, in view of \eqref{eq:linear} and the constructions, we can get the following system of linear  equations: 
\begin{equation}\label{Smn1}
\left\{\begin{array}{ll}
 0=a_{mn-m+1}+a_{mn-m+2}+\cdots +a_{mn-3}-a_{mn-2}-a_{mn-1},\\
 0=a_{mn-m+1}+a_{mn-2}+a_{mn-1},\\
 0=-a_{mn-m+1}+a_{mn-m+2},\\
 0=-a_{mn-m+2}+a_{mn-m+3},\\
 \vdots \\
 0=-a_{mn-5}+a_{mn-4},\\
 0=-a_{mn-4}+a_{mn-3}+a_{mn-2},\\
 0=-a_{mn-3}-a_{mn-2}+a_{mn-1},\\
 0=-a_{mn-m+1}-a_{mn-m+2}-\cdots -a_{mn-3}-a_{mn-1}.
\end{array}\right.
\end{equation}
Here, we want to indicate that only the first two and the last three  equations from \eqref{Smn1} show up when $m=5$. 
From \eqref{Smn1}, it is easy to see that 
$$
a_{mn-m+1}=a_{mn-m+2}=\cdots =a_{mn-1}=0.
$$
Besides, similar as before, we can get 
$$
a_{mn-2m+2}=a_{mn-2m+3}=\ldots=a_{mn-m}=0.
$$
 Clearly, $\vb_1,\vb_2,\ldots,\vb_{mn-2m+1}$ are linearly independent, this completes the proof of linear independence.

By Lemma \ref{Cassels}, there exists a basis $\{\wb_1,\wb_2,\ldots,\wb_{mn-1}\}$ of $\Lp$ such that
$$
\vb_i=a_{i1}\wb_1+\cdots+a_{ii}\wb_i
$$
with integers $0\le a_{ij}<a_{ii}$ $(j<i; i=1,2,\ldots,mn-1)$. From the proof of Theorem \ref{basis2}, it follows directly that
$$
\vb_i=\wb_i \quad \textrm{for $1\le i \le mn-m$}.
$$ 

Fixing an arbitrary integer $k$ with $mn-m+1\le k \le mn-3$ and in view of the constructions in \eqref{basis1}, we suppose that the first entry $-1$ of $\vb_k$ corresponds to Part $j$. Then, we consider the sum of entries of $\vb_i$ ($1\le i \le k$) corresponding to Part $j$. This sum is zero if $i<k$, and it is equal to -1 if $i=k$. This implies that 
$$
a_{ii=1} \fand \vb_i= \wb_i \quad \textrm{for $mn-m+1\le i \le mn-3$}.
$$ 

Now, let $k=mn-2$, we consider   
$$
\vb_k=a_{k1}\vb_1+\cdots+a_{k,k-1}\vb_{k-1}+a_{kk}\wb_k.
$$
 We compute the sum of entries of $\vb_i$ ($1\le i \le k$) corresponding to Part 0 and Part $m-1$. This sum is zero if $i<k$, and it is equal to -1 if $i=k$. This implies that 
$$
a_{kk}=1 \fand \vb_{mn-2}= \wb_{mn-2}.
$$ 

Finally, let $k=mn-1$, we consider   
$$
\vb_k=a_{k1}\vb_1+\cdots+a_{k,k-1}\vb_{k-1}+a_{kk}\wb_k.
$$ 
We put $b_j=a_{k,mn-m+j}$ for $1\le j \le m-1$. For each $i$ $(0\le i \le m-1)$, we assume that the sum of entries of $\wb_k$ corresponding to Part $i$ is $x_i$. Then, we can get a system of linear equations exactly like \eqref{system}:
\begin{equation}\label{Smn2}
\left\{\begin{array}{ll}
 -1=b_{1}+b_{2}+\cdots +b_{m-3}-b_{m-2}+b_{m-1}x_0,\\
 1=b_{1}+b_{m-2}+b_{m-1}x_1,\\
 0=-b_{1}+b_{2}+b_{m-1}x_2,\\
 0=-b_{2}+b_{3}+b_{m-1}x_3,\\
 \vdots \\
 0=-b_{m-5}+b_{m-4}+b_{m-1}x_{m-4},\\
 0=-b_{m-4}+b_{m-3}+b_{m-2}+b_{m-1}x_{m-3},\\
 1=-b_{m-3}-b_{m-2}+b_{m-1}x_{m-2},\\
 -1=-b_{1}-b_{2}-\cdots -b_{m-3}+b_{m-1}x_{m-1}.
\end{array}\right.
\end{equation}
Thus, as before we can obtain 
$$
\vb_{mn-1}=\wb_{mn-1},
$$
which completes the proof of the theorem.
\end{proof}

To complete the proof of Theorem \ref{main}, it remains to consider some special cases.
\begin{theorem}\label{basis5}
 Assume that $\PP$ satisfies $\PP \cong \Z/m\Z \times \Z/n\Z$ with $m\le n \le 4$ and $mn\ge 5$. Then, $\Lp$ has a basis of minimal vectors.
\end{theorem}
\begin{proof}
By the previous discussions, we only need to consider three cases: $m=2$ and $n=4$, $m=3$ and $n=3$, and $m=4$ and $n=4$.

By calculating the determinants with some computer algebra systems (like PARI/GP \cite{Pari}) and comparing with Proposition \ref{det}, one can easily check that the following is true.

If $m=2$ and $n=4$, the following is a required basis of $\Lp$:
\begin{equation}\label{basis24}
\left\{\begin{array}{ll}
  \vb_1=(1,1,-1,-1,0,0,0,0),\\
  \vb_2=(0,0,0,0,1,-1,-1,1),\\
  \vb_{3}=(1,1,0,0,-1,-1,0,0),\\
  \vb_{4}=(1,0,1,0,-1,0,-1,0),\\
  \vb_{5}=(1,0,0,1,-1,0,0,-1),\\
  \vb_6=(0,1,1,0,0,-1,-1,0),\\
  \vb_7=(1,-1,0,0,1,0,0,-1).
\end{array}\right.
\end{equation}

If $m=3$ and $n=3$, the following is a required basis of $\Lp$:
\begin{equation}\label{basis33}
\left\{\begin{array}{ll}
  \vb_1=(1,1,0,-1,0,0,0,-1,0),\\
  \vb_2=(1,1,0,0,-1,0,-1,0,0),\\
  \vb_{3}=(1,1,0,0,0,-1,0,0,-1),\\
  \vb_{4}=(1,0,1,-1,0,0,0,0,-1),\\
  \vb_{5}=(1,0,1,0,-1,0,0,-1,0),\\
  \vb_6=(1,0,1,0,0,-1,-1,0,0),\\
  \vb_7=(0,1,1,-1,0,0,-1,0,0),\\
  \vb_8=(1,0,0,-1,-1,0,0,1,0).
\end{array}\right.
\end{equation}

If $m=4$ and $n=4$, the following is a required basis of $\Lp$:
\begin{equation}\label{basis44}
\left\{\begin{array}{ll}
  \vb_1=(1,1,-1,-1,0,0,0,0,0,0,0,0,0,0,0,0),\\
  \vb_2=(0,0,0,0,1,1,-1,-1,0,0,0,0,0,0,0,0),\\
  \vb_{3}=(0,0,0,0,0,0,0,0,1,-1,-1,1,0,0,0,0),\\
  \vb_{4}=(0,0,0,0,0,0,0,0,0,0,0,0,1,-1,-1,1),\\
  \vb_{5}=(1,1,0,0,0,0,0,0,-1,-1,0,0,0,0,0,0),\\
  \vb_6=(1,0,1,0,0,0,0,0,-1,0,-1,0,0,0,0,0),\\
  \vb_7=(1,0,0,1,0,0,0,0,-1,0,0,-1,0,0,0,0),\\
  \vb_8=(0,1,1,0,0,0,0,0,0,-1,-1,0,0,0,0,0),\\
  \vb_9=(0,0,0,0,1,1,0,0,0,0,0,0,-1,-1,0,0),\\
  \vb_{10}=(0,0,0,0,1,0,1,0,0,0,0,0,-1,0,-1,0),\\
  \vb_{11}=(1,-1,0,0,1,0,0,-1,0,0,0,0,0,0,0,0),\\
  \vb_{12}=(1,-1,0,0,0,0,0,0,1,0,0,-1,0,0,0,0),\\
  \vb_{13}=(1,0,0,0,1,0,0,0,-1,0,0,0,-1,0,0,0),\\
  \vb_{14}=(1,0,0,0,0,-1,-1,0,0,0,0,1,0,0,0,0),\\
  \vb_{15}=(1,0,0,0,-1,-1,0,0,0,1,0,0,0,0,0,0).
\end{array}\right.
\end{equation}

\end{proof}

\section{Determinant}

For the lattice $\Lp$ attached to the elliptic curve $E$, it is easy to compute its  determinant.

\begin{proposition}\label{det}
If $|\PP|\ge 2$, then we have $\det \Lp= |\PP|^{3/2}$.
\end{proposition}
\begin{proof}
Let $n=|\PP|$. First, we have $[A_{n-1}:\Lp]=n$ by definition. Besides, it is well-known that $\det A_{n-1}=\sqrt{n}$; see \cite[Proposition 4.2.2]{Martinet} (one should note that the definition there is slightly different from ours). Thus, we get
$$
\det \Lp=n \det A_{n-1}=n^{3/2}.
$$
\end{proof}

The famous Minkowski-Hlawka theorem asserts that for every $k> 1$ there exists a lattice $L$ of rank $k$ such that
\begin{equation}\label{Minkowski}
\Delta(L)\ge \zeta(k)/2^{k-1}.
\end{equation}
However, all the current proofs of this theorem are non-constructive. It is still not known how to construct lattices with packing densities satisfying \eqref{Minkowski} for arbitrary $k$. Here, we can find that the lattices $\Lp$ can provide several  examples with the help of computer.

\begin{proposition}
Assume that $4\le |\PP| \le 47$. Then, $\Lp$ satisfies \eqref{Minkowski}.
\end{proposition}

We want to remark that the above conclusion is no longer true when $|\PP| > 47$.

\section{Covering radius}

Let $L$ be a lattice in the Euclidean space $V:=\spa_{\R}L$. The \textit{covering radius} of $L$, denoted by $\mu(L)$, is defined as the smallest real number $r$ such that any point in $V$ is within distance $r$ from the lattice $L$. That is
$$
\mu(L)=\max_{\vb\in V}\min_{\x \in L} \| \vb-\x \|.
$$
 A \textit{deep hole} with respect to $L$ is a point $\vb\in V$ such that $\min_{\x \in L} \| \vb-\x \|=\mu(L)$.

In \cite[Theorem 3.4]{Lenny2014}, Fukshansky and Maharaj gave the following upper bound:
\begin{equation}\label{upper0}
\mu(\Lp)\le \frac{1}{2}(\sqrt{n^2+4n+8}+\sqrt{n}),
\end{equation}
where $n=|\PP|$. Here, we will improve this estimate substantially.

We first recall a construction in \cite[Proof of Theorem 3.4]{Lenny2014}.             Suppose that $\PP=\{\Pb_0,\Pb_1,\ldots,\Pb_{n-1}\}$ as described in Section \ref{preliminary} such that $\Pb_0=\Ob$. Given $\vb=(a_0,a_1,\ldots,a_{n-1}) \in A_{n-1}$, since $\PP$ is a group, there exists some $j$ ($0\le j \le n-1$) such that $a_0\Pb_0+a_1\Pb_1+ \cdots + a_{n-1}\Pb_{n-1}=\Pb_j$. Now, define
\begin{equation}\label{tilde}
\tilde{\vb}=\left\{\begin{array}{ll}
 (a_0+1,a_1,\ldots,a_{j-1},a_j-1,a_{j+1},\ldots,a_{n-1}) & \textrm{if $j\ne 0$,} \\
 \vb & \textrm{if $j=0$.}
\end{array}\right.
\end{equation}
Here, we want to indicate that if $j=1$, then the vector $\tilde{\vb}=(a_0+1,a_1-1,a_2,\ldots,a_{n-1})$. 
Thus, the vector $\tilde{\vb}\in \Lp$ by Theorem \ref{principal}, and $\|\vb - \tilde{\vb}\|\le \sqrt{2}$.

For the convenience of the reader, we first recall some basic results about $A_{n-1}$  without proof; see Section 6.1 of Chapter 4 in \cite{Conway}.

\begin{lemma}\label{An}
For $n\ge 2$, we have
\begin{equation}
\mu(A_{n-1})=\left\{\begin{array}{ll}
 \frac{1}{2}\sqrt{n} & \textrm{if $n$ is even,} \\
 \frac{1}{2}\sqrt{n-1/n} & \textrm{if $n$ is odd.}
\end{array}\right.
\notag
\end{equation}
In particular, let $i$ be the integer part of $n/2$ and $j=n-i$, a typical deep hole  with respect to $A_{n-1}$ is
\begin{equation}
\wb=\left\{\begin{array}{ll}
 (\frac{1}{2},\ldots,\frac{1}{2},-\frac{1}{2},\ldots,-\frac{1}{2}) & \textrm{if $n$ is even,} \\
 (\frac{1}{2}-\frac{1}{2n},\ldots,\frac{1}{2}-\frac{1}{2n},-\frac{1}{2}-\frac{1}{2n},\ldots,-\frac{1}{2}-\frac{1}{2n}) & \textrm{if $n$ is odd,}
\end{array}\right.
\notag
\end{equation}
with $j$ positive entries and $i$ negative entries.
\end{lemma}

Now, we give sharp bounds for the covering radius $\mu(\Lp)$.

\begin{theorem}\label{radius}
For $n=|\PP|\ge 2$, we have
$$
\mu(A_{n-1}) \le \mu(\Lp)\le \mu(A_{n-1})+\sqrt{2}.
$$
\end{theorem}
\begin{proof}
Put $V:=\spa_{\R}A_{n-1}=\spa_{\R}\Lp$. By definition and noticing $\Lp \subseteq A_{n-1}$, we have
$$
\mu(\Lp)=\max_{\vb\in V}\min_{\x \in \Lp} \| \vb-\x \|\ge \max_{\vb\in V}\min_{\x \in A_{n-1}} \| \vb-\x \|=\mu(A_{n-1}),
$$
which gives the desired lower bound.

Now we want to get the upper bound. By \cite[Theorem 3.4]{Lenny2014}, we have
$$
\max_{\vb\in A_{n-1}}\min_{\x \in \Lp} \| \vb-\x \|\le \sqrt{2}.
$$
On the other hand, for any $\vb\in A_{n-1}$ and $\x \in \Lp$, since all the entries of $\vb-\x$ are integers and $\vb-\x\in A_{n-1}$, we must have $\| \vb-\x \|\ge \sqrt{2}$ if $\vb \ne \x$. So, we actually have
$$
\max_{\vb\in A_{n-1}}\min_{\x \in \Lp} \| \vb-\x \|= \sqrt{2}.
$$

Given an arbitrary point $\vb\in V$, we can pick a point $\wb\in A_{n-1}$ such that $\|\vb-\wb\|\le \mu(A_{n-1})$; then as \eqref{tilde} we define a point $\tilde{\wb}\in \Lp$ satisfying $\|\wb-\tilde{\wb}\|\le \sqrt{2}$. Thus, we have
$$
\|\vb-\tilde{\wb}\|\le \|\vb-\wb\|+\|\wb-\tilde{\wb}\|\le \mu(A_{n-1})+\sqrt{2},
$$
which implies the desired upper bounds.
\end{proof}

\begin{corollary} \label{upper}
For $n=|\PP|\ge 2$, we have
\begin{equation}
\mu(\Lp)\le \left\{\begin{array}{ll}
 \frac{1}{2}\sqrt{n}+\sqrt{2} & \textrm{if $n$ is even,} \\
 \frac{1}{2}\sqrt{n-1/n}+\sqrt{2} & \textrm{if $n$ is odd.}
\end{array}\right.
\notag
\end{equation}
\end{corollary}

It is clear to see that the upper bound in Corollary \ref{upper} is much better than the bound \eqref{upper0} given in \cite{Lenny2014}.

As mentioned before, an improvement of the upper bound in Corollary  \ref{upper} in the case of the Barnes lattices is given in \cite{Bottcher}, that is 
$$
\mu(\Lp)<\frac{1}{2}\sqrt{n+4\log (n-2)+6-4\log 2+10/(n-1)},
$$
where $\PP$ is a cyclic group of order $n\ge 3$; see \cite[Theorem 1.3]{Bottcher}.

Finally, let $V=\spa_{\R}\Lp$, and $\wb\in V$ defined as in Lemma \ref{An}. Then, we have
$$
\min_{\x \in \Lp} \| \wb-\x \| \ge \min_{\x \in A_{n-1}} \| \wb-\x \|=\mu(A_{n-1}).
$$
On the other hand, let $\x_0$ be the zero vector. Since $\|\wb-\x_0\|=\mu(A_{n-1})$ and $\x_0\in \Lp$, we have
$$
\min_{\x \in \Lp} \| \wb-\x \| \le \mu(A_{n-1}).
$$
So, we get 
$$
\min_{\x \in \Lp} \| \wb-\x \| = \mu(A_{n-1}).
$$
In view of Theorem \ref{radius}, we can view $\wb$ as an approximation of a deep hole  with respect to $\Lp$.

\section*{Acknowledgement}
The author would like to thank Prof. Igor E. Shparlinski for introducing him this topic. He wants to thank Prof. Heinz-Georg Quebbemann for sending him a copy of \cite{Quebbemann}. He is also grateful to Prof. Lenny Fukshansky and Prof. Hiren Maharaj for their interest on this paper and  
helpful discussions. He also appreciates Prof. Lenny Fukshansky sending him the recent preprint \cite{Bottcher}. The research of the author was supported by the Australian Research Council Grant DP130100237.
Finally, he wants to thank the referee for careful reading and very useful comments.


\begin{thebibliography}{99}

\bibitem{Barnes}
E.S. Barnes, \textit{The perfect and extreme senary forms}, Canad. J. Math. \textbf{9} (1957), 235--242.

\bibitem{Bottcher}
A. B\"ottcher, L. Fukshansky, S.R. Garcia and H. Maharaj, \textit{On lattices generated by finite Abelian groups}, preprint, 2014, \url{http://arxiv.org/abs/1406.7595}.

\bibitem{Conway1995}
J.H. Conway and N.J.A. Sloane, \textit{A lattice without a basis of minimal vectors}, Mathematika \textbf{42} (1995), 175--177.

\bibitem{Conway}
J.H. Conway and N.J.A. Sloane, \textit{Sphere packings, lattices and groups}, 3rd edition, Springer-Verlag, 1999.

\bibitem{Cassels}
J.W.S. Cassels, \textit{An introduction to the geometry of numbers}, Springer-Verlag, 1971.

\bibitem{Lenny2014}
L. Fukshansky and H. Maharaj, \emph{Lattices from elliptic curves over finite fields}, Finite Fields Appl. \textbf{28} (2014), 67--78.

\bibitem{Gruber}
P.M. Gruber and C.G. Lekkerkerker, \textit{Geometry of numbers}, 2nd edition, North-Holland Publishing Co., 1987.

\bibitem{Martinet}
J. Martinet, \textit{Perfect lattices in Euclidean spaces}, Springer-Verlag, 2003.

\bibitem{Pari}
PARI/GP, version {\tt 2.7.0}, Bordeaux, 2014, \url{http://pari.math.u-bordeaux.fr/}.

\bibitem{Quebbemann}
H.-G. Quebbemann, \textit{Lattices from curves over finite fields}, preprint, 1989.

\bibitem{Tsfasman1990}
M.Y. Rosenbloom and M.A. Tsfasman, \emph{Multiplicative lattices in global fields}, Invent. Math. \textbf{101} (1990), 687--696.

\bibitem{Ruck1987}
H.-G. R\"uck, \textit{A note on elliptic curves over finite fields}, Math. Comp. \textbf{49}(179)  (1987), 301--304.

\bibitem{Schoof}
R. Schoof, \textit{Nonsingular plane cubic curves over finite fields}, 
J. Combin. Theory Ser. A \textbf{46} (1987), 183--211. 

\bibitem{Silverman}
J.H. Silverman, \textit{The arithmetic of elliptic curves}, 2nd edition,
Springer, Dordrecht, 2009.

\bibitem{Tsfasman1985}
M.A. Tsfasman, \textit{Group of points of an elliptic curve over a finite field}, in: Theory of numbers and its applications, Tbilisi, 1985,  286--287 (in Russian).

\bibitem{Tsfasman1991}
M.A. Tsfasman and S.G. Vl\u{a}du\c{t}, \emph{Algebraic-geometric codes}, Kluwer Academic Publishers, 1991.

\bibitem{Tsfasman2007}
M.A. Tsfasman, S.G. Vl\u{a}du\c{t} and D. Nogin, \textit{Algebraic geometric codes: basic notions}, Amer. Math. Soc., 
Providence, RI, 2007.

\bibitem{Voloch}
J.F. Voloch, \textit{A note on elliptic curves over finite fields},
Bull. Soc. Math. France \textbf{116} (1988), 455--458.




\end{thebibliography}
\end{document}